\documentclass[12pt,a4paper]{amsart}
\usepackage{amsfonts,amsmath,amsthm,amssymb}
\usepackage{pstricks,pst-node}
\usepackage{fullpage}

\newtheorem*{thm}{Theorem}

\title{A short proof of a famous combinatorial identity}

\author{Rui Duarte}
\address{Center for Research and Development in Mathematics and Applications,
Department of Mathematics,
University of Aveiro}
\email{rduarte@ua.pt}

\author{Ant\'onio Guedes de Oliveira}
\address{CMUP and Mathematics
  Department, Faculty of Sciences, University of Porto}
\email{agoliv@fc.up.pt}

\thanks{The work of both authors was supported in part by the European
  Regional Development Fund through the program COMPETE - Operational
  Programme Factors of Competitiveness (``Programa Operacional
  Factores de Competitividade'') and by the Portuguese Government
  through FCT - Funda\c{c}\~ao para a Ci\^encia e a Tecnologia, under
  the projects PEst-C/MAT/UI0144/2011 and PEst-C/MAT/UI4106/2011.}

\date{\today}

\begin{document}

\begin{abstract}
  We explain how the identity
$$\sum_{i+j=n}\binom{2i}{i}\binom{2j}{j}\;=\;\displaystyle4^n$$
is an easy consequence of the inclusion-exclusion principle.
\end{abstract}

\maketitle

\noindent
There is a famous combinatorial identity on the convolution of the
central binomial coefficients:
\begin{equation}\label{mainId}
\sum_{i+j=n}\binom{2i}{i}\binom{2j}{j}\;=\;\displaystyle4^n\,.
\end{equation}
This identity can be immediately proven by squaring both sides of the
equality
$$(1-4x)^{-1/2}=\sum_{n\geq0}\binom{2n}{n}\,x^n\,,$$
which is an easy application of Newton's generalized binomial theorem
\cite[Exercises 1.2.c and 1.4.a]{RS}.
Combinatorial proofs of the identity also exist, being the first one
credited by Paul~Erd\H{o}s to Gy\"orgy~Haj\'os, in the thirties of the
twentieth century \cite{MS}, as well as a number of interesting
non-bijective proofs that appeared in the literature from various
sources and viewpoints, even quite recently \cite{VdA,ChX}.

But we do not know, as to now, of a proof both self-contained and
really short of this identity. Yet, in here, we present a short and
elementary proof ---based on the inclusion-exclusion principle--- of
\emph{a generalization} of \eqref{mainId} (see \cite{DGO} for related
identities). So, we suggest that, perhaps, we could not see the forest
for the trees.

\begin{thm}
  For every nonnegative integer numbers $i$, $j$ and $n$ and 
  every real number $\ell$,
$$\sum_{i+j=n}\binom{2i-\ell}{i}\binom{2j+\ell}{j}\;=\;4^n\,.$$
\end{thm}
\begin{proof}

\newcommand\A{\mathcal{A}}

We first prove that, for any value of $\ell$ and any nonnegative
integer $p$,
\begin{equation}\label{aux}
\sum_{i=0}^p(-1)^i \binom{\ell-i}{p}\binom{p}{i}\;=\;1.
\end{equation}
Note that it is sufficient to prove the result for the integers
$\ell>2p$, since the sum on the left-hand member defines a polynomial
(thus constant) in $\ell$. Now, consider the collection $\A$ of the
subsets of $\{1,\dotsc,\ell\}$ with $\ell-p$ elements and let $\A_j$
be the set of elements of $\A$ that contain $j$, for
$j=1,\dotsc,p$. Then, what \eqref{aux} says is that
$\{p+1,\dotsc,n\}\in\A$ is the unique element of $\A$ that does not
contain any integer between $1$ and $n$.
In more detail,
by the inclusion-exclusion principle,
\begin{align*}
\binom{\ell}{\ell-p}-1=  |\A_1\cup\dotsb\cup\A_p|&=
\sum_{i=1}^p (-1)^{i+1}\bigg(\sum_{1\leq j_1<\dotsb<j_i\leq p}|
\A_{j_1}\cap\dotsb\cap \A_{j_i}|\bigg)\\
&=\sum_{i=1}^p (-1)^{i+1}\binom{p}{i}\,\binom{\ell-i}{p}\end{align*}
---and we have \eqref{aux}--- since, for any integers
  $j_1,\dotsc,j_i$ such that $1\leq j_1<\dotsb<j_i\leq p$, 
$$|A_{j_1}\cap\dotsb\cap A_{j_i}|=\binom{\ell-i}{\ell-i-p}=\binom{\ell-i}{p}
\,.$$
\medskip

Now, by definition, Vandermonde's identity and since
$\binom{\ell-1-i}{i}\binom{\ell-1-2i}{p-i}=\binom{\ell-1-i}{p}\binom{p}{i}$,
\begin{align*}
\sum_{i+j=n} \binom{2i-\ell}{i} \binom{2j+\ell}{j}
& = \sum_{i+j=n}(-1)^i \binom{\ell-1-i}{i} \binom{2n+\ell-2i}{j} \\
& = \sum_{i+j=n} (-1)^i \binom{\ell-1-i}{i} \left[\sum_{k=0}^j
\binom{2n+1}{k} \binom{\ell-1-2i}{j-k} \right] \\
& = \sum_{k=0}^n \left[ \binom{2n+1}{k} \sum_{i=0}^{n-k} (-1)^i
\binom{\ell-1-i}{i} \binom{\ell-1-2i}{n-k-i} \right]\\
& = \sum_{k=0}^n \binom{2n+1}{k}\\
& = \frac{1}{2}\sum_{k=0}^{2n+1}\binom{2n+1}{i}\\
& = 4^n\,.\\[-20pt]
\end{align*}
\end{proof}


\end{document}